\documentclass[11pt,reqno]{amsart}

\usepackage[usenames]{color}
\usepackage{graphicx}
\usepackage{amscd}
\usepackage[colorlinks=true,
linkcolor=webgreen,
filecolor=webbrown,
citecolor=webgreen]{hyperref}

\definecolor{webgreen}{rgb}{0,.5,0}
\definecolor{webbrown}{rgb}{.6,0,0}
\usepackage{amsthm}
\usepackage{fullpage}
\usepackage{enumerate}
\usepackage{epstopdf}
\usepackage{graphics,amsmath,amssymb}
\usepackage{graphicx}

\usepackage{amsfonts}
\usepackage{latexsym}
\usepackage{url}

\newcommand{\seqnum}[1]{\href{http://oeis.org/#1}{\underline{#1}}}

\theoremstyle{plain}
\newtheorem{theorem}{Theorem}[section]
\newtheorem{lemma}[theorem]{Lemma}
\newtheorem{corollary}[theorem]{Corollary}
\newtheorem{proposition}[theorem]{Proposition}

\theoremstyle{definition}

\theoremstyle{remark}

\newcommand{\tr}{{\rm tr}}
\newcommand\bigzero{\makebox(0,0){\text{\huge0}}}

\def\adots{
  \mathinner{\mkern1mu\raise1pt\hbox{.}\mkern2mu\raise4pt\hbox{.}
  \mkern2mu\raise7pt\vbox{\kern7pt\hbox{.}}\mkern1mu}}

\begin{document}

\begin{center}

 \title[Matrices in the Hosoya Triangle]
 {Matrices in the Hosoya Triangle}
 \author{Matthew Blair}
 \thanks{ Several of the main results in this paper were found by the first author while working on his undergraduate research project under the guidance of the second and third authors}
 \address{Department of Mathematical Sciences\\
		The Citadel\\
		Charleston, SC \\
		U.S.A.}
 \email{blairm@citadel.edu}
 \author{Rigoberto Fl\'orez}
 \address{Department of Mathematical Sciences\\
		The Citadel\\
		Charleston, SC \\
		U.S.A.}
  \email{rflorez1@citadel.edu}
 \author{Antara Mukherjee}
 \address{Department of Mathematical Sciences\\
		The Citadel\\
		Charleston, SC \\
		U.S.A.}
 \email{mukherjeea1@citadel.edu}
\end{center}

\begin{abstract}
In this paper we use well-known  results from linear algebra as tools to explore some properties of products of Fibonacci numbers.
Specifically, we explore the behavior of the  eigenvalues, eigenvectors, characteristic polynomials, determinants, and the norm
of non-symmetric matrices embedded in the Hosoya triangle. We discovered that most of these objects either embed again in the
Hosoya triangle or they give rise to Fibonacci identities.

We also study the nature of these matrices when their entries are taken $\bmod$ $2$. As a result, we found an infinite family
of non-connected graphs. Each graph in this family has a complete graph with loops attached to each of  its vertices as a
component and the other components are isolated vertices. The Hosoya triangle allowed us to show the beauty of both, the algebra and geometry.
\end{abstract}

\maketitle

\section {Introduction}

The \emph{Hosoya triangle} \cite{hosoya}, is a triangular array where the entries are products of Fibonacci numbers.
Our main purpose of this paper is to show some connection between this triangle and several aspects of geometry of the Hosoya triangle,
graph theory, and the Fibonacci number sequence,  all of them by means of linear algebra techniques.
Several authors have been geometrically representing  the beauty of some elementary concepts of  algebra, number theory,
and combinatorics bridging them with Fibonacci numbers using the Hosoya triangle
\cite{Bondarenko,florezHiguitaJunesGCD, florezjunes, florezHiguitaMukherjeeGeometry, griffiths, hosoya, koshy,koshy2}.

The recurrence relations of Fibonacci numbers provide an interesting way to study properties of matrices with these entries. For example,
in 2002 Lee et. al  \cite{lee} studied matrices that have squares of Fibonacci numbers   in the diagonal and the rest of the entries
are generalized Fibonacci numbers.  In particular, they studied the Cholesky factorizations and the  eigenvalues of these matrices.
In 2008 Stanimirovi\'{c} et. al. \cite{stanimirovic} studied the inverses of generalized Fibonacci and Lucas matrices.

In this paper we study the nature of non-symmetric matrices embedded in the Hosoya triangle, (matrices with product of Fibonacci
numbers as entries). We use well-known results in linear algebra as tools to explore  different patterns within this triangle.
We present three infinite families of matrices where the eigenvalues and eigenvectors satisfy a ``closure property"
in the set of Fibonacci numbers and partially in the Hosoya triangle. We classify those three families in rank one matrices,
skew-triangular matrices, and antidiagonal matrices.

The first family (matrices of rank one) have the feature that the matrices are products of
two vectors $\mathbf{u}$ and $\mathbf{v}^{T}$. The entries of the vectors are consecutive Fibonacci numbers ---in fact, the vectors
$\mathbf{u}$ and $\mathbf{v}$ are  located on the sides of the Hosoya triangle. The matrices of this family have exactly one non-zero
eigenvalue which is a combination of Lucas and Fibonacci numbers. We would like to emphasize that these matrices are
diagonalizable (with exactly one non-zero eigenvalue) where the entries of the eigenvectors are again points of the Hosoya triangle.
For instance, the vector $\mathbf{u}$  is one of those eigenvectors.

We connect the first family of matrices with graph theory by observing the fractal generated by the Hosoya triangle $\bmod \; 2$.
Therefore, the matrices of our first family give rise to the adjacency matrices of undirected and non-connected  graphs.  The graphs consist of
 a complete graph with loops attached to each of its vertices as one component and some isolated vertices as the other components.

The second family satisfies that the non-zero eigenvalues are Fibonacci numbers convolved with themselves
(see \cite{Czabarka, Moree} or \cite{sloane} at \seqnum{A001629}). The matrices of the third family have the feature
that the product of their eigenvalues (that are not necessarily integers) is a product of Fibonacci numbers.

\subsection{Hosoya triangle} \label{CoordinateSystem}
In this section we give the classic definition of the Hosoya triangle denoted by $\mathcal{H}$.

The \emph{Hosoya sequence} $\left\{H(r,k)\right\}_{r,k> 0}$ is defined using the double recursion
\begin{equation}\label{Hosoya:Seq}
 H(r,k)= H(r-1,k)+H(r-2,k)  \; \text{ and } \;
 H(r,k)= H(r-1,k-1)+H(r-2,k-2)
 \end{equation}
with initial conditions
$H(1,1)= H(2,1)= H(2,2)= H(3,2)=1,$
where $1\le k \le r$. For brevity, we write $H_{r,k}$ instead of $H(r,k)$ throughout the paper.

This sequence gives rise to the \emph{Hosoya triangle}, where the entry in position $k$
(taken from left to right) of the $r${th} row is equal to $H_{r,k}$ see Table \ref{tabla1} (and also
\cite{sloane} at \seqnum{A058071}). For simplicity, in this paper we use $\mathcal{H}$ to denote the Hosoya triangle.
One can also refer to \cite{florezHiguitaJunesGCD, hosoya, koshy,koshy2} for the definition of the Hosoya triangle.
Another way to represent each entry  (or point) of the Hosoya triangle is $H_{r,k}= F_kF_{r-k+1}$ for positive integers
$r$ and $k$ with $1\le k\le r$  (see \cite{florezHiguitaJunesGCD,koshy}).  It is easy to
see that an $n$th \emph{diagonal} (either slash or backslash diagonal) in $\mathcal{H}$ is the collection of all
Fibonacci numbers multiplied by $F_n$.

\begin{table} [!ht] \small
\centering
\addtolength{\tabcolsep}{-3pt} \scalebox{.85}{
\begin{tabular}{ccccccccccccc}
&&&&&&                                                               $H_{1,1}$                                                 &&&&&&\\
&&&&&                                                  $H_{2,1}$     &&     $H_{2,2}$                                        &&&&&\\
&&&&                                         $H_{3,1}$    &&     $H_{3,2}$     &&     $H_{3,3}$                               &&&&\\
&&&                                 $H_{4,1}$   &&     $H_{4,2}$     &&     $H_{4,3}$      &&    $H_{4,4}$                     &&&\\
&&                        $H_{5,1}$     &&     $H_{5,2}$    &&     $H_{5,3}$     &&     $H_{5,4}$     &&     $H_{5,5}$             &&\\
&            $H_{6,1}$     &&    $H_{6,2}$    &&     $H_{6,3}$     &&     $H_{6,4}$      &&    $H_{6,5}$     &&    $H_{6,6}$    & \\
$H_{7,1}$     &&    $H_{7,2}$    &&     $H_{7,3}$     &&     $H_{7,4}$      &&    $H_{7,5}$     &&    $H_{7,6}$    &&    $H_{7,7}$     \\
&&&&&&&&&&&&\\
\end{tabular}}
\caption{Hosoya triangle  $\mathcal{H}$.} \label{tabla1}
\end{table}

\section{Non-symmetric matrices of rank one in the Hosoya triangle}

In this section we study the first family where every matrix is of rank one (matrices that are products of
two vectors). We first define matrices using the backslash diagonals of $\mathcal{H}$ and then present one of the main results of this paper. In this result we
give a  closed formula for the trace of these matrices. We also present a result on the eigenvectors and the eigenvalues of matrices found in the Hosoya triangle.

For $m, n$, and $t$, all positive integers with $m,t \leq n$, we define in \eqref{eq2} the\emph{ backslash matrix} $B(m,n,t)$
(the \emph{slash matrix} is defined similarly). Let $s=(t-1)$ and $r_i=(m+n-1)+i$ for $i=0,1,2,\ldots, s$, then
\begin{equation} \label{eq2}
B(m,n,t) =  \left[ {\begin{array}{lllll}
   H_{r_{0},m} & H_{r_{0}-1,m} & H_{r_{0}-2,m} & \cdots& H_{r_{0}-s,m} \\
   H_{r_{1},m+1}  &  H_{r_{1}-1,m+1}& H_{r_{1}-2,m+1} & \cdots & H_{r_{1}-s,m+1}\\
   \vdots &\vdots&\vdots&\ddots &\vdots\\
   H_{r_{s},m+s}  &  H_{r_{s}-1,m+s} & H_{r_{s}-2,m+s} & \cdots & H_{r_{s}-s,m+s}
  \end{array} } \right].
\end{equation}

For example,  Figure  \ref{Blair_matrix} Part (a) and Part (b) depicts  $B(3,7,5)$ and $B(1,7,7)$, respectively.
Note that the first entry of  $B(m,n,t)$  is the point in the intersection of
the $m$-th backslash diagonal and the $n$-th slash diagonal.
 In particular, the entry (point) in position $(1,1)$ of $B(3,7,5)$  (which is represented by $H_{r_{0},m}$) can be determined by
 writing $r_{0}=9$ and $m=3$ and using that $H_{r,k}=F_{k}F_{r-k+1}$.
Therefore,  $H_{9,3}=F_3F_7=26$. This technique may be used to find all entries of the matrix.

\begin{figure}[!ht]
	\includegraphics[scale=0.27]{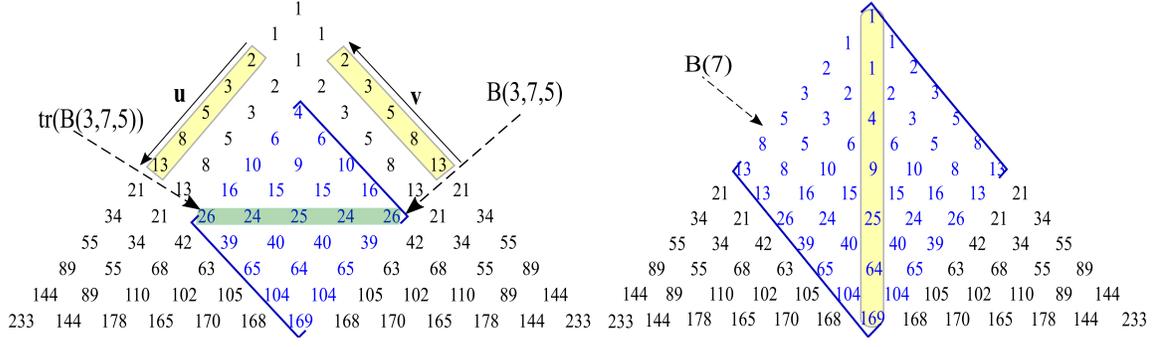}
\caption{(a) Matrix $B(3,7,5)$ in Hosoya triangle.\hspace{0.2cm} (b) Matrix $B(7)$ in Hosoya triangle.}\label{Blair_matrix}
\end{figure}
\subsection{Diagonalizable Matrices in Hosoya triangle}
In this section we describe some properties of diagonalizable matrices in the Hosoya triangle.

We use $\tr(B(m,n,t))$ to denote the trace of $B(m,n,t)$ throughout this section. The formula given in Corollary \ref{corollary:generalizedFibonacci}
generalizes the convolutions given in \cite{Czabarka, Moree} and \cite[A06733]{sloane}.

\begin{lemma}\label{Lemma1}
Let $m, n$,  and $t$ be fixed positive integers. If $L_{k}$ is the $k$th Lucas number, then the following hold
\begin{enumerate}[(a)]
	\item \label{SumProdFibo} $\displaystyle\sum_{i=0}^{t}F_{m+i}F_{n-i}=\left[(t+1)L_{m+n}-\sum_{i=0}^{t}(-1)^{n-i}L_{m-n+2i}\right]/5$ and
	\item \label{AltSumProdFibo} $\displaystyle\sum_{i=0}^{t-1}(-1)^{n-i-1}L_{m-n+2i}=(-1)^{n-t} F_{m-n+2t-1}+(-1)^{m-1} F_{n-m+1}$.
\end{enumerate}
\end{lemma}

\begin{proof}
We prove Part \eqref{SumProdFibo}, the proof of Part \eqref{AltSumProdFibo} follows easily using mathematical induction on $t$ and we omit it.
Simplifying the expression $F_{m+i}F_{n-i}$ using the Binet formula for Fibonacci numbers \cite{koshy} we obtain
	$$\sum_{i=0}^{t}F_{m+i}F_{n-i}=\left[(t+1)L_{m+n}-\sum_{i=0}^{t}(-1)^{n-i}L_{m-n+2i}\right]/5.$$
This completes the proof of Part \eqref{SumProdFibo}.
\end{proof}

\begin{proposition}\label{proposition:generalizedFibonacci}
If $m, n, t$ are fixed positive integers with $m,t\le n$ and $L_{k}$ represents Lucas numbers for $k\ge 0$ then
$$\tr(B(m,n,t))= \left[t L_{m+n}+(-1)^{n-t} F_{m-n+2t-1}+(-1)^{m-1} F_{n-m+1}\right]/5.$$
\end{proposition}

\begin{proof}
From \eqref{eq2} we have that
$\tr(B(m,n,t))= \sum_{i=0}^{t-1} H_{r_{i}-i,m+i},$
 where $r_i=(m+n-1)+i$ for $i=0,1,2,\ldots, (t-1)$. Since $H_{r,k}=F_{k}F_{r-k+1}$, we have
 $\tr(B(m,n,t))= \sum_{i=0}^{t-1} F_{m+i}F_{n-i}$.
 The conclusion follows from Lemma \ref{Lemma1}.
 \end{proof}

\begin{corollary}\label{corollary:generalizedFibonacci}
If $m, n$, and $t$ are fixed positive integers, then
$$\sum_{i=0}^{t-1} F_{m+i} F_{n-i} = \left[t L_{m+n}+(-1)^{n-t} F_{m-n+2t-1}+(-1)^{m-1} F_{n-m+1}\right]/5.$$
\end{corollary}

Let $A$ be an $n\times n$ matrix. Using  linear algebra techniques it is easy to verify the following statement.
\[A \text{ is of rank 1 if and only if } A=\mathbf{u}\cdot \mathbf{v}^{T} \text{ for some non-zero column vectors  } \mathbf{u}, \mathbf{v} \in \mathbb{R}^{n}.\]
Moreover, the characteristic equation of $A$ is given by $x^{n-1}(x-\tr(A))=0$ and that $\mathbf{u}$ is the
eigenvector associated with $\tr(A)$ (where $\tr(A)$ denotes the trace of the matrix $A$).
Notice that the multiplicity of $x=0$ is $(n-1)$, therefore there are $(n-1)$
eigenvectors associated with $x=0$. Finally, $\mathbf{u}$ is not orthogonal to $\mathbf{v}$.

Now we establish certain notations needed to prove Proposition \ref{Proposition:eigen}. Let $W=\{u,v_{1},v_{2},\ldots,v_{t-1}\}$
be the set of eigenvectors of $B(m,n,t)$, then $v_{i}$ is orthogonal to $u$ for all $i$. Let $E_{j}(x)$ indicate the elementary
matrix obtained by multiplying the $j$-th row of the $t\times t$ identity matrix by $x$.

\begin{proposition}\label{Proposition:eigen}
 Let $B^{\prime}=B(m,n,t)$ be a backslash matrix embedded in the Hosoya triangle $\mathcal{H}$. Then the following hold

\begin{enumerate}[(1)]
\item \label{eigenvalues:B} the eigenvalues of $B^{\prime}$ are $\lambda=0$ with algebraic multiplicity $(t-1)$
                            and $\lambda=\tr(B^{\prime})$ with algebraic multiplicity one.
\item \label{Diagaliz:B} The matrix $B^{\prime}$ is diagonalizable and if $s=t-1$ then the eigenvectors of $B^{\prime}$ are given by,
\[
    \mathbf{u} = \begin{bmatrix}
           F_{m} \\
           F_{m+1} \\
           F_{m+2}\\
           \vdots \\
           F_{m+s}
         \end{bmatrix},
         \mathbf{v_{1}} = \begin{bmatrix}
           -F_{n-1} \\
           F_{n} \\
           0\\
           \vdots \\
           0
         \end{bmatrix}, \dots, \mathbf{v_{s}} = \begin{bmatrix}
           -F_{n-s} \\
           0 \\
           0\\
           \vdots \\
           F_{n}
         \end{bmatrix}.
  \]
\end{enumerate}
\end{proposition}

\begin{proof}
By the definition of $B^{\prime}$ above,  $H_{r,k}=F_kF_{r-k+1}$ using \eqref{eq2} we have

\[
B^{\prime}=  \left[ {\begin{array}{lllll}
   F_{m}F_{n} & F_{m}F_{n-1}& F_{m}F_{n-2} & \cdots& F_{m}F_{n-s} \\
   F_{m+1}F_{n}  &  F_{m+1}F_{n-1}& F_{m+1}F_{n-2} & \cdots & F_{m+1}F_{n-s}\\
   \vdots &\vdots&\vdots&\ddots &\vdots\\
   F_{m+s} F_{n} &  F_{m+s} F_{n-1}& F_{m+s}F_{n-2} & \cdots & F_{m+s}F_{n-s}\\
  \end{array} } \right]
.\]

Now it is easy to verify that $B^{\prime}=\mathbf{u}\mathbf{v}^T$ where $\mathbf{u}$ is as given in the statement of
Part \eqref{Diagaliz:B} and $\mathbf{v}^T=[ F_{n},  F_{n-1},  F_{n-2}, \dots,  F_{n-s}]$.
This and the application of the linear algebra techniques described above completes the proof of Part \eqref{eigenvalues:B}.

Proof of Part \eqref{Diagaliz:B}. It is clear that $\mathbf{u}$ is an eigenvector associated with $\lambda=\tr(B^{\prime})$
(see the discussion above). In order to find the eigenvectors associated with $\lambda=0$ , it is enough to find a
basis for the null space of $B^{\prime}$. Since
\[E_{s}\left(\frac{1}{F_{m+s}}\right)E_{s-1}\left(\frac{1}{F_{m+s-1}}\right)\cdots E_{2}\left(\frac{1}{F_{m+1}}\right)E_{1}\left(\frac{1}{F_{m}}\right) B^{\prime}
= \left[ {\begin{array}{ccccc}
   F_{n} & F_{n-1} & F_{n-2} & \cdots& F_{n-s} \\
   F_{n}  &  F_{n-1} & F_{n-2} & \cdots & F_{n-s}\\
   \vdots &\vdots&\vdots&\ddots &\vdots\\
   F_{n}  &  F_{n-1}& F_{n-2} & \cdots & F_{n-s}\\
  \end{array} } \right],\]
  from this last matrix  it is easy to see that  $\{ \mathbf{v_{1}},  \mathbf{v_{2}}, \dots,  \mathbf{v_{s}}\}$
  is a basis for the null space of $B^{\prime}$. This completes the proof of Part \eqref{Diagaliz:B}. \end{proof}

Notice, in general, if in the definition given in \eqref{Hosoya:Seq} instead of $H_{1,1}= H_{2,1}= H_{2,2}=H_{3,1}=1,$
we consider $H_{1,1}=a^2; H_{2,1}=ab;  H_{2,2}=ab; H_{3,2}=b^2$ with $a, b \in \mathbb{Z}$
(see examples in \cite{sloane}, \seqnum{A284115, A284129, A284126, A284130, A284127, A284131, A284128}) then the matrix $B(m,n,t)$
defined over this general recursive sequence  satisfies the following properties: first $B(m,n,t)$ has rank 1; second,
there are vectors  {\bf u} and {\bf v} from the sides of the general triangle such that $B(m,n,t)={\bf u} \cdot {\bf v}^T$
and third, that the only non-zero eigenvalue of the matrix is given by  the trace $\tr(B(m,n,t))$ which is equal to
${\bf u} \cdot {\bf v}^T$ (similar to the vectors {\bf u} and {\bf v} shown in Figure \ref{Blair_matrix}(a)).

A particular case of the matrices $B(m,n,t)$ when $m=1$ and $t=n$ are persymmetric matrices
(matrices that are symmetric with respect to the antidiagonal), we denote these matrices
simply by $B(n)$ (or for brevity by $B$). Therefore, persymmetric matrices are square matrices
that are symmetric along the skew-diagonal (see Figure \ref{Blair_matrix}(b)).

If $\mathbf{u}$, $\mathbf{v_{i}}$ for $i=1, \dots, n-1$ are the eigenvectors of $B$
(as given in Proposition \ref{Proposition:eigen} Part \eqref{Diagaliz:B}), then the matrix of eigenvectors of $B$ is
\begin{equation}\label{Q:Eigenvect}
Q_{n}=[\mathbf{u}, \mathbf{v_{1}}, \dots, \mathbf{v_{n-1}}].
\end{equation}
If $\tr(B)$ is the non-zero eigenvalue associated to $\mathbf{u}$ (recall that the eigenvalues associated to
$\mathbf{v_{i}}$ for $i=1, \dots, n-1$ are all zero), then the diagonal matrix of $B$ is $D_{n}=[d_{ij}]$,
where the only non-zero entry is $d_{11}=\tr(B)$.

\begin{corollary}\label{Cor:persymmetric}
If $Q_{n}$ and $B(n)$ (or $B$) are as described above, then for $k>0$ the following  holds
\[B^{k}Q_{n}=Q_{n}\left(\frac{nL_{n+1}+2F_{n}}{5}\right)^{k}.\]
\end{corollary}

\begin{proof}
From Proposition \ref{Proposition:eigen} Part \eqref{Diagaliz:B} we know that $B$ is diagonalizable.
This and \eqref{Q:Eigenvect} imply that $B=Q_{n} D_{n} Q_{n}^{-1}$. Therefore, $B^{k}=Q_{n} D^{k}_{n} Q_{n}^{-1}$.
Therefore, $B^{k}Q_{n}=Q_{n}(\tr(B))^kI_{n}$, where $I_{n}$ is the $n\times n$ identity matrix.
To complete the proof recall from \cite{Czabarka,Moree} that
$\tr(B)=\sum_{i=1}^{n}F_{n-i}F_{i}=(nL_{n+1}+2F_{n})/5.$
\end{proof}

\subsection{Normal matrices}

Suppose that $A=B^{T}B$, where $B$ is the persymmetric matrix $B(n)$. Note that $A$ has rank one and therefore it
has exactly one non-zero eigenvalue $\lambda$. So, the spectral radius of $A$ is $\rho (A) = \lambda$. From linear algebra
we know that $||A||_{2}=\sqrt{\rho(A^TA)}=\sqrt{\lambda}$ (see \cite{Golub,Zhang}).  The identity in Proposition \ref{norm:result}
Part \eqref{norm:result5} is  well-known. However, this tells us that adding the square of every entry of the matrix gives the
sum of all points in the antidiagonal which is again a point of the  Hosoya triangle. The identity in Proposition
\ref{norm:result} Part \eqref{norm:result6} is also a well-known identity, but this tells us that the sum of all entries in
the matrix is the difference of two points in the Hosoya triangle. Another interpretation is given by recalling that the
matrix norm of $A$ measures  how much a vector $X$ can be extended by applying matrix $A$ on it. Applying these concepts
to the matrices in the Hosoya triangle, we obtain that the norm is again a point within the triangle. So, this norm  provides
a good geometric interpretation in the Hosoya triangle of two well-known Fibonacci identities. One can refer back to
Figures \ref{Blair_matrix}~(b) and  \ref{Blair_matrix2}~(b) to explore the geometric significance of the value of those
norms. The proof of Proposition \ref{norm:result} Part  \eqref{norm:result2} gives the geometry of the identity
given in Part \ref{norm:result4} (it is easy to see that $A$ is a normal).
Proposition \ref{norm:result} Part \ref{norm:result7} shows that the singular value of $B$ (see \cite{Golub}) is again an entry of the Hosoya triangle.

 Note that if $C=B\circ B$ is the Hadamard product \cite{Horn}
of the matrix $B$ with itself, then the only non-zero eigenvalue of $C$ is $\lambda=\tr{(C)}$.

\begin{proposition}\label{norm:result}
	If $A=B^{T}B$, then
	\begin{enumerate}[(1)]
		\item \label{norm:result1} $A$ has exactly one non-zero eigenvalue $\lambda$.
		\item  \label{norm:result2} If $\lambda$ is the non-zero eigenvalue of $A$, then $\lambda$ is the product of the sum of antidiagonal elements of
		$B$ with the sum of
		antidiagonal elements of $B^{T}$. Thus, $\lambda=\left(\sum_{i=1}^{n}H_{2i+1,i}\right)^{2}=\left(F_{n}F_{n+1}\right)^{2}$.
		\item  \label{norm:result3} The eigenvalue $\lambda =\tr(A)=\sum_{i,j=1}^{n}\left(F_{i}F_{j}\right)^2=\left(F_{n}F_{n+1}\right)^{2}$.
		\item \label{norm:result5} If $\lambda$ is the non-zero eigenvalue of $A$, then
		$$\sqrt{\lambda}=||B||_{2}=F_{n}F_{n+1}=\label{norm:result4} \sqrt{\sum_{i,j=1}^{n}\left(F_{i}F_{j}\right)^2}=\sum_{i=1}^{n}F_{i}^2.$$
		\item \label{norm:result6} $\sqrt{\sum_{i,j=1}^{n} F_{i}F_{j}}=\sum_{i=1}^{n}F_{i}=F_{n+2}-1=||B||_{\infty}/F_{n}.$
		\item \label{norm:result7}  $\sqrt{\lambda}=H_{2n,n}$.
	\end{enumerate}
\end{proposition}

\begin{proof}
The proofs of Parts \eqref{norm:result1}, \eqref{norm:result3}, and \eqref{norm:result5}--\eqref{norm:result7}, are straightforward
using linear algebra techniques mentioned in the above comments, therefore we omit those details. Since $A$ is a normal matrix the proof of Part
\eqref{norm:result4} follows  from the Schur inequality \cite{Zhang} (an alternative proof can be found using parts \eqref{norm:result2} and \eqref{norm:result3} or basic algebra).

Proof of Part  \eqref{norm:result2}. We know that
$A=B^{T}B=(\mathbf{u}\cdot\mathbf{v}^{T})^{T}(\mathbf{u}\cdot\mathbf{v}^{T})=\mathbf{v}\cdot (\mathbf{u}^{T}\cdot \mathbf{u})\cdot \mathbf{v} ^{T}$.
Since, $(\mathbf{u}^{T}\cdot \mathbf{u})$ is a real number we have $A=(\mathbf{u}^{T}\cdot \mathbf{u})\mathbf{v} \cdot \mathbf{v} ^{T}$.
Therefore, the eigenvalues of $A$ are actually the eigenvalues of $\mathbf{v} \cdot \mathbf{v} ^{T}$ multiplied by $\mathbf{u}^{T}\cdot \mathbf{u}$.
We know that the non-zero eigenvalue of $\mathbf{v}\cdot \mathbf{v}^{T}$ is given by
$\tr(\mathbf{v}\cdot \mathbf{v}^{T})$.

Since \[\mathbf{u}^{T}\cdot \mathbf{u}=[F_{n}\, F_{n-1}\,\cdots\,F_{2}\, F_{1}]\begin{bmatrix}
           F_{n} \\
           F_{n-1} \\
           F_{n-2}\\
           \vdots \\
           F_{1}
         \end{bmatrix} =\sum_{i=1}^{n}F_{i}^{2}=F_{n}F_{n+1}  \text{ and } \tr(\mathbf{v}\cdot \mathbf{v}^{T})=\sum_{i=1}^{n}F_{i}^{2}=F_{n}F_{n+1},\]
         we have that the eigenvalue $\lambda$ of $A$ is $\left(F_{n}F_{n+1}\right)^{2}$. This completes the proof of part (2).
\end{proof}

\subsection{Graphs in the Hosoya triangle}

Several authors have been interested in graphs generated by considering the Pascal triangle entries $\bmod \; 2$. The first example is the well-known Sierpin\'ski
triangle. Other examples can be found in  Koshy \cite[Chapter 9]{koshy2}. Here he discussed Pascal graphs in the Pascal Binary Triangle
(see also \cite{deo, BakerSwart, romik}). We use a similar procedure for a family of  matrices in $\mathcal{H}$.

We consider the adjacency matrix constructed by taking each entry of  the persymmetric matrix  $B(n)$ modulo $2$ where $ n\equiv 2 \bmod  3$.
This gives rise to a  family of adjacency matrices of undirected  and non-connected  graphs.  The graphs are composed of a complete graph with loops
attached to each of its vertices as a component and the  other components are some isolated vertices (see Table \ref{tabla2}).

\begin{proposition}
If $k\ge 0$ and $n=3k+2$, then the graph of the adjacency matrix corresponding to $B(n) \bmod 2$ is a
complete graph on $2(k+1)$ vertices with loops at every vertex and $k$ isolated vertices.
\end{proposition}

\begin{proof} It is known that the Fibonacci number $F_{n}\equiv 0 \bmod 2 \iff 3\mid n$. This and the definition of $B(n)$ imply that every
third row and every third column of $B(n)$ are formed by even numbers and that the remaining rows and columns are formed by odd numbers only.
Thus, if $b_{ij}$ is an entry of $B(n)$, then $b_{ij}\equiv 0 \bmod 2 \iff i \equiv 0 \mod 3 \text{ or } j \equiv 0 \mod 3$.
This and $n=3k+2$ imply that $B(n) \bmod 2 $ contains $k$ columns and $k$ rows with zeros as entries. The remaining 2(k+1) rows and
columns have ones as entries. These two features of $B(n) \bmod 2 $  give us a complete graph on $2(k+1)$ vertices with loops at every vertex and $k$
isolated vertices. This completes the proof.
\end{proof}

\begin{table}[!ht]
\begin{tabular}{|l|l|c|} \hline
$3k+2$ &  Matrix  \hspace{6.5cm}       Graph  \\ \hline  \hline
2   &
        \begin{tabular}{p{6.5cm}c}
            {$\displaystyle
                {\renewcommand{\arraystretch}{1.2}
                \begin{pmatrix}
                    1 & 1 \\
 		   1 & 1 \\
                \end{pmatrix}}
            $}
            &
            $\vcenter{\hbox{\includegraphics[scale=0.3]{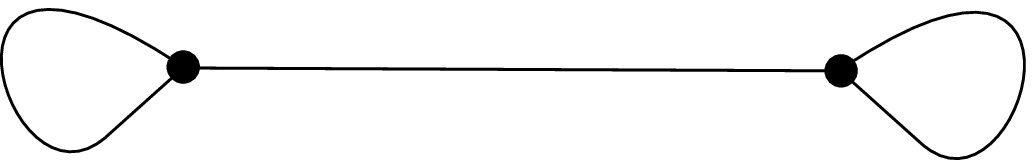}}}$
        \end{tabular}
     \\ [20pt]\hline
5   &
        \begin{tabular}{p{6.5cm}c}
            {$\displaystyle
                {\renewcommand{\arraystretch}{1.2}
                \begin{pmatrix}
 1 & 1 & 0 & 1 & 1 \\
 1 & 1 & 0 & 1 & 1 \\
 0 & 0 & 0 & 0 & 0 \\
 1 & 1 & 0 & 1 & 1 \\
 1 & 1 & 0 & 1 & 1 \\
                \end{pmatrix}}
            $}
            &
            $\vcenter{\hbox{\includegraphics[scale=0.3]{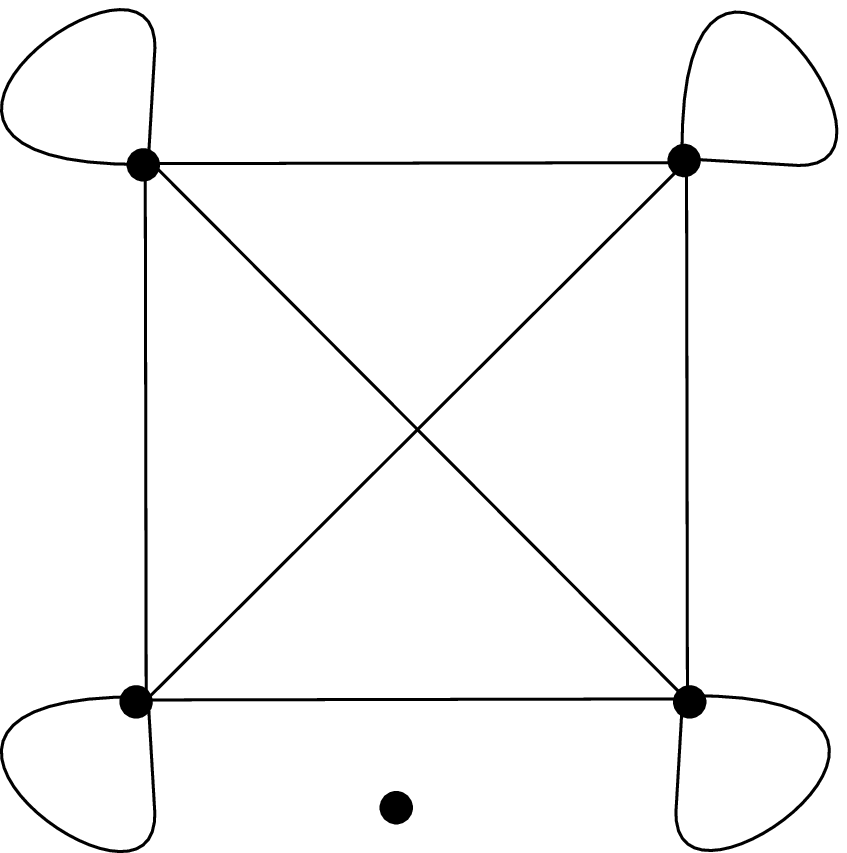}}}$
        \end{tabular}
     \\  [20pt] \hline

8   &
        \begin{tabular}{p{6.5cm}c}
            {$\displaystyle
                {\renewcommand{\arraystretch}{1.2}
                \begin{pmatrix}
 1 & 1 & 0 & 1 & 1 & 0 & 1 & 1 \\
 1 & 1 & 0 & 1 & 1 & 0 & 1 & 1 \\
 0 & 0 & 0 & 0 & 0 & 0 & 0 & 0 \\
 1 & 1 & 0 & 1 & 1 & 0 & 1 & 1 \\
 1 & 1 & 0 & 1 & 1 & 0 & 1 & 1 \\
 0 & 0 & 0 & 0 & 0 & 0 & 0 & 0 \\
 1 & 1 & 0 & 1 & 1 & 0 & 1 & 1 \\
 1 & 1 & 0 & 1 & 1 & 0 & 1 & 1 \\
                \end{pmatrix}}
            $}
            &
            $\vcenter{\hbox{\includegraphics[scale=0.3]{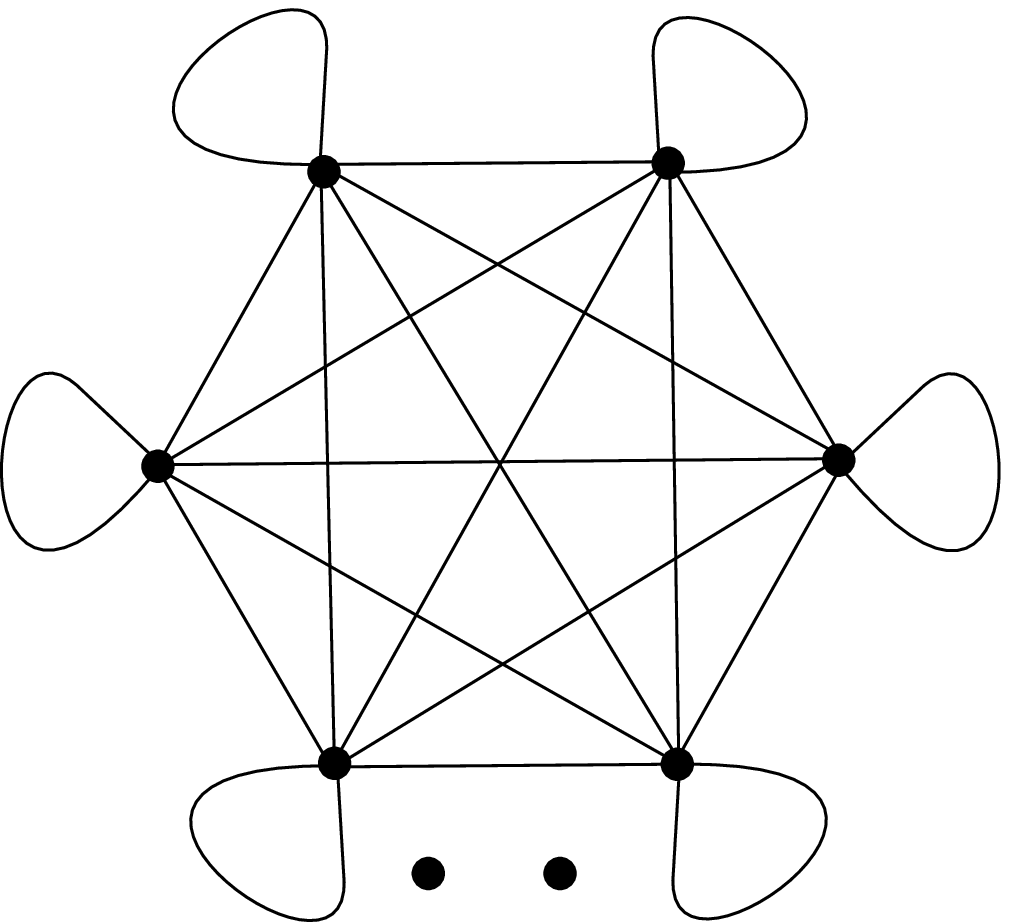}}}$
        \end{tabular}
     \\  [20pt]\hline
11   &
        \begin{tabular}{p{7cm}c}
            {$\displaystyle
                {\renewcommand{\arraystretch}{1.2}
                \left( \begin{array}{ccccccccccc}
 1 & 1 & 0 & 1 & 1 & 0 & 1 & 1 & 0 & 1 & 1 \\
 1 & 1 & 0 & 1 & 1 & 0 & 1 & 1 & 0 & 1 & 1 \\
 0 & 0 & 0 & 0 & 0 & 0 & 0 & 0 & 0 & 0 & 0 \\
 1 & 1 & 0 & 1 & 1 & 0 & 1 & 1 & 0 & 1 & 1 \\
 1 & 1 & 0 & 1 & 1 & 0 & 1 & 1 & 0 & 1 & 1 \\
 0 & 0 & 0 & 0 & 0 & 0 & 0 & 0 & 0 & 0 & 0 \\
 1 & 1 & 0 & 1 & 1 & 0 & 1 & 1 & 0 & 1 & 1 \\
 1 & 1 & 0 & 1 & 1 & 0 & 1 & 1 & 0 & 1 & 1 \\
 0 & 0 & 0 & 0 & 0 & 0 & 0 & 0 & 0 & 0 & 0 \\
 1 & 1 & 0 & 1 & 1 & 0 & 1 & 1 & 0 & 1 & 1 \\
 1 & 1 & 0 & 1 & 1 & 0 & 1 & 1 & 0 & 1 & 1 \\
\end{array} \right)}
            $}
            &
            $\vcenter{\hbox{\includegraphics[scale=.25]{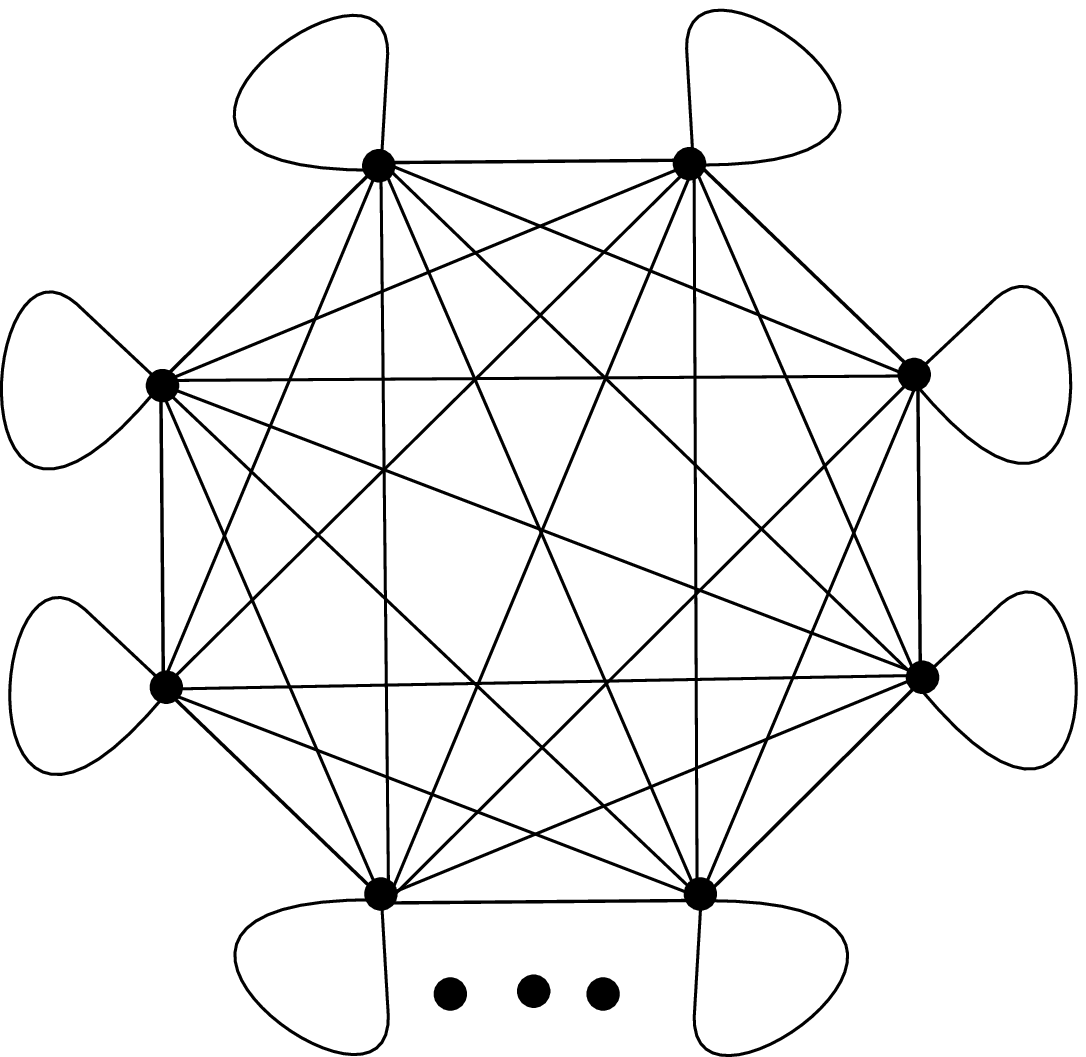}}}$
        \end{tabular}
\\ \hline
\end{tabular}
\caption{Adjacency Graphs of $B(n) \bmod 2$ in  $\mathcal{H}$.} \label{tabla2}
\end{table}

\section{Antidiagonal matrices and skew-triangular matrices in the Hosoya triangle}\label{antidiagonal}

In this section we study a family of antidiagonal matrices where the entries of the antidiagonal are points from
the ``median" of $\mathcal{H}$ (see Figure \ref{Blair_matrix2}~(a)).  Let $A$ be a $n\times n$ in this family. We prove that the eigenvalues of $A$ are again entries of $\mathcal{H}$ as well as
the entries of its eigenvectors (except maybe by the sign of those entries).
The eigenvectors of $A$ form the rows of a new square matrix $E$ where non-zero entries of $E$ are in the diagonal and antidiagonal.
The diagonal of $E$ is formed by all points in a horizontal line of $\mathcal{H}$, while the antidiagonal of $E$ is the same
antidiagonal of $A$ seen in Figure \ref{Blair_matrix2}~(b). Note that every first entry of a row of $A$ is located in the $n$th backslash diagonal of
$\mathcal{H}$, while every first entry of a row of $E$ is located in the first backslash diagonal of $\mathcal{H}$.

The matrix $E$ can be seen geometrically as a cross in $\mathcal{H}$ where the only non-zero entries
of $E$ are the first $n$ entries of the ``median" of $\mathcal{H}$ and the entries of the $n$th row of $\mathcal{H}$.
However, some eigenvectors of $A$ have negative entries, but in $\mathcal{H}$ all entries are positive, so our representation is not a perfect geometric representation. Therefore, to have a good geometric representation of the eigenvectors of $A$
we introduce a convention (only for this type of eigenvector). We are going to assume that negative entries of
the eigenvector of $A$ are represented by points on the left side of the ``median" of $\mathcal{H}$.

\subsection{Eigenvalues of antidiagonal matrices} First  we formally define the matrix $A$ in the following way:

 \begin{equation}\label{Def:Mat:A}
 A=[a_{ij}]_{1\le i,j\le n} \quad \text{ where } \quad a_{ij}=\begin{cases}
      F_{i}^2, & \text{ if } j=n-i+1;\\
     0, & \text{ otherwise} .
   \end{cases}
\end{equation}

\begin{figure}[!ht]
	\includegraphics[width=15cm]{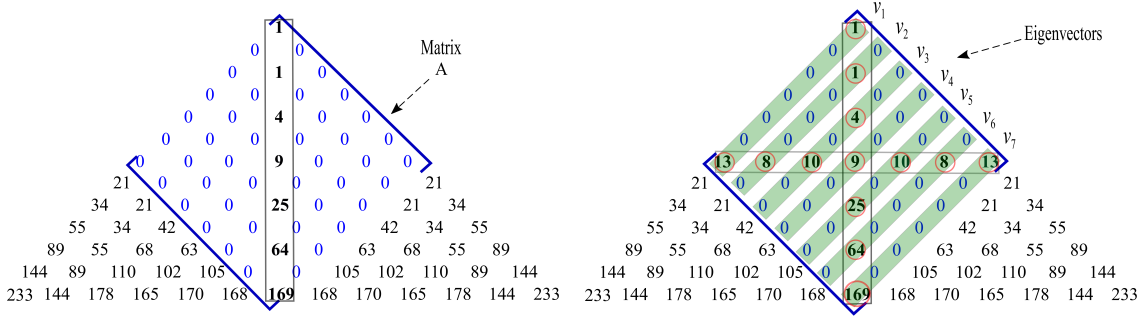}
\caption{(a) Anti-diagonal Matrix $A$.\hspace{3.3cm} (b) Eigenvectors of $A$.}\label{Blair_matrix2}
\end{figure}

\begin{proposition}\label{proposition:antidiagonal} If $A$ is a matrix as defined in \eqref{Def:Mat:A}, then the following hold

\begin{enumerate}
  \item \label{Prop:Ant:P1} the rank of $A$ equals $n$.
  \item \label{Prop:Ant:P2} If $n=2k$  and $1\le i\le k$, then the eigenvalues of $A$ are
   \[ \lambda_{j}=\begin{cases}
      \quad F_{i}F_{n-i+1}, & \text{ if } j=i;\\
     -F_{i}F_{n-i+1}, & \text{ if } j=k+i.
   \end{cases}
\]
\item \label{Prop:Ant:P3} If $n=2k-1$, then the eigenvalues of $A$ are $\lambda_{k}=F_{k}^2$ and for $1\le i< k$
   \[ \lambda_{j}=\begin{cases}
      \quad F_{i}F_{n-i+1}, & \text{ if } j=i;\\
     -F_{i}F_{n-i+1}, & \text{ if } j=k+i.
   \end{cases}
\]
  \item \label{Prop:Ant:P4} If $n=2k$, then the eigenvectors of $A$ are $x_{i}=[\alpha_{1},\alpha_{2},\ldots, \alpha_{n}]$ for $1\le i\le k$, where
   \[ \alpha_{t}=\begin{cases}
      \pm F_{j}, & \text{ if } t=j \text{ and } j=i;\\
     \pm F_{n-j+1}, & \text{ if } t=n-j+1 \text{ and } i=j;\\
      0, & \text{ if } i\ne j.
   \end{cases}
\]
    \item \label{Prop:Ant:P5} If $n=2k-1$, then the eigenvectors of $A$ are $y_{i}=[\beta_{1},\beta_{2},\ldots, \beta_{n}]$
    for $1\le i\le k-1$ and $y_{k}=[0,0,\ldots,1,0\ldots,0]$, where

       \[ \beta_{t}=\begin{cases}
      \pm F_{j}, & \text{ if } t=j \text{ and } j=i;\\
     \pm F_{n-j}, & \text{ if } t=n-j \text{ and } i=j;\\
     0,  & \text{ if } i\ne j.
   \end{cases}
\]

\end{enumerate}
\end{proposition}
\begin{proof}
The proof of Part \eqref{Prop:Ant:P1} is straightforward since the rank of the column space of $A$ is $n$. We now prove Part \eqref{Prop:Ant:P2},
the proof of Part \eqref{Prop:Ant:P3}  is similar and it is omitted. Observe that if $n=2k$ and $I_{n}$ is the identity matrix, then
$\pm  F_{i}F_{n-i+1}$ for $1\le i\le k$ are the  solutions of the characteristic equation $\det(A-\lambda I_{n})=0$. Thus,
$\lambda_{j}= F_{i}F_{n-i+1}$ if $j=i$ and $\lambda_{j}= -F_{i}F_{n-i+1}$ if $j=k+i$.

We now prove Part \eqref{Prop:Ant:P4},  the proof of Part \eqref{Prop:Ant:P5} is similar and it is omitted.
The eigenvector corresponding to the eigenvalue $\lambda=F_{i}F_{n-i+1}$ for $1\le i\le k$ can be
found by solving the system of equations
$(A-F_{i}F_{n-i+1}I_{n})\mathbf{x}=\mathbf{0}$, where $\mathbf{x}=[x_{1},x_{2},\ldots,x_{n}]^{T}$ and $\mathbf{0}$ is the zero vector.

Using row operations on $(A-F_{i}F_{n-i+1}I_{n})$, the system of equations given above simplifies to
\[x_{i}-(F_{i}/F_{n-i+1})x_{n-i+1}=0 \text{ and } x_{j}=0 \text{ when } i\ne j.\]
Therefore, the eigenvector corresponding to the eigenvalue  $F_{i}F_{n-i+1}$ is given by $[\alpha_{1},\alpha_{2},\ldots, \alpha_{n}]$, where
   \[ \alpha_{t}=\begin{cases}
       F_{j}, & \text{ if } t=j \text{ and } i=j;\\
      F_{n-j+1}, & \text{ if } t=n-j+1 \text{ and } i=j;\\
      0, & \text{ if } i\ne j.
   \end{cases}
\]
The analysis for the eigenvector associated to $\lambda=-F_{i}F_{n-i+1}$ is similar. This completes the proof.
\end{proof}

{\bf Example.} If we have the matrix $A$ seen in Figure \ref{Blair_matrix2}~(a), then the eigenvalues of $A$ are
$\pm F_{1}F_{7}$, $\pm F_{2}F_{6}$, $\pm F_{3}F_{5}$, and $F_{4}^2$.
  The eigenvectors of $A$ are given by [0,0, 0,1,0, 0, 0] and

  \[\begin{cases}
 [F_{1}^2,0,0,0,0,0,F_{7}F_{1}]\\
 [F_{1}^2,0,0,0,0,0,-F_{7}F_{1}]
\end{cases},
\begin{cases}
[0,F_{2}^2,0,0,0,F_{6}F_{2},0]\\
 [0,F_{2}^2,0,0,0,-F_{6}F_{2},0],
\end{cases}
\begin{cases}
[0,0,F_{3}^2,0,F_{5}F_{3},0,0]\\
 [0,0,F_{3}^2,0,-F_{5}F_{3},0,0].
\end{cases}
\]

Comment: The matrix $A$ is diagonalizable. In fact, when $n=2k$, we have

\[\left[ {\begin{array}{ccccccc}
  0 & 0 &  \cdots&  0 & F_{1}^2\\
  0 &  0 &  \cdots& F_{2}^2 & 0\\
   \vdots &\vdots& \adots & \vdots &\vdots\\
  F_{n}^2 &  0  &  \cdots & 0 &0\\
  \end{array} } \right] P=P\left[ {\begin{array}{cccccc}
  F_{1}F_{2k} & 0 & 0 & \cdots& 0&0\\
   0 &  -F_{1}F_{2k} & 0 & \cdots & 0&0\\
   \vdots &\vdots& \vdots& \ddots &\vdots&\vdots\\
  0 &  0 & 0& \cdots &F_{k}F_{k+1}&0\\
   0 &  0 & 0& \cdots &0&-F_{k}F_{k+1}\\
  \end{array} } \right] \]
where

\[P=\left[ {\begin{array}{cccccc}
  F_{1} &F_{1} & 0 &0&  \cdots& 0\\
   0 &0&  F_{2}&F_{2}&  \cdots & 0\\
   \vdots &\vdots&\vdots &\vdots&\vdots & \vdots\\
   0 & 0&0 &0 &\cdots &  F_{k}\\
    0 & 0&0 &0 &\cdots & F_{k+1}\\
    \vdots& \vdots &\vdots& \vdots& \vdots & \vdots\\
   0  &0      & F_{2k-1} & -F_{2k-1}&   \cdots &  0  \\
  F_{2k} &-F_{2k}  & 0 &0& \cdots &0\\
  \end{array} } \right]. \]
Similarly, it is easy to verify that $A$ is diagonalizable when $n=2k-1$.

  Another interesting property of $A$ is that if
   \[A =\left[ {\begin{array}{ccccc}
                0 & 0 &  \cdots&  0 & F_{1}^2\\
                0 &  0 &  \cdots& F_{2}^2 & 0\\
   \vdots &\vdots& \adots & \vdots &\vdots\\
  F_{n}^2 &  0  &  \cdots & 0 &0\\
  \end{array} } \right],
  \text{ then }
  A^{-1}=\left[ {\begin{array}{ccccc}
                0 & 0 &  \cdots&  0 & \frac{1}{F_{n}^2}\\
                0 &  0 &  \cdots& \frac{1}{F_{n-1}^2} & 0\\
   \vdots &\vdots& \adots & \vdots &\vdots\\
  \frac{1}{F_{1}^2} &  0  &  \cdots & 0 &0\\
  \end{array} } \right].\]

 The properties we describe for the antidiagonal matrix $A$ are true in general. Thus, if
  $A^{\prime}=[a_{i,j}]_{1\le i,j\le n}$ is an antidiagonal matrix with the entries $a_{i,j}\in \mathbb{R}$ ---such that the following roots make sense in the set of real numbers---
 then the eigenvalues of $A^{\prime}$ are
   \[ \lambda_{i}=\begin{cases}
     \pm \sqrt{a_{i,n}a_{n-i+1,i}},\quad & \text{ if } n=2k;\\
     \pm \sqrt{a_{i,n}a_{n-i,i}}, \quad  &  \text{ if } n=2k-1.
   \end{cases}
\]
 The eigenvectors have the same behavior as in the previous case. For example, if $n=2k$, then the eigenvectors of
 $A$ are $x_{i}=[\alpha_{1},\alpha_{2},\ldots, \alpha_{2k}]$, where
  $\alpha_{j}=\pm a_{i,n+1-j}$ and $\alpha_{n+1-j}=\pm a_{n+1-j,i}$ if $j=i$ and $\alpha_{j}=0$ if $i\ne j$ for $1\le i,j\le k$.

  Note that if $A$ is an antidiagonal matrix (see definition \eqref{Def:Mat:A})  then the characteristic polynomial of $A$ is given by

   \[p(x)=\begin{cases}
 \;\; \displaystyle\prod_{i=1}^{n}\left(x^2 - F_{i}^2F_{n-i+1}^2\right), & \text{ when } n \text{ is even };\\
  \displaystyle \prod_{i=1}^{(n-1)/2}\left(x^2- F_{i}^2F_{n-i+1}^2\right)\left(x-F_{(n+1)/2}^{2}\right), & \text{ when } n \text{ is odd } .
  \end{cases}
  \]

  In fact, using Vieta's formula for polynomials we can say that the characteristic polynomial for $A$ is a polynomial of degree $n$ where the coefficient of $x^{n-1}$ is $\tr(A)$, the coefficient of $x^{0}$ is $\det(A)$ (determinant of $A$), and the coefficient of $x^{n-k}$ is given by
   \[(-1)^{k}\sum_{1\le {i_{1}}\le\ldots\le {i_{k}}\le n}\lambda_{i_{1}}\lambda_{i_{2}}\ldots \lambda_{i_{k}} \text{ where } \lambda_{i_{j}} \text{ represent the eigenvalues of } A.\]

\subsection{Determinants of skew-triangular matrices in the Hosoya triangle}\label{determinants}

In this section we define skew-triangular matrices in $\mathcal{H}$. This family of matrices does not necessarily have integers as eigenvalues.
So, we analyze their determinants to obtain some information about their eigenvalues.  We also discuss some properties of the determinants (see
the determinant in \eqref{antitriangular:mat} on page \pageref{antitriangular:mat}). The determinant of an antidiagonal matrix is well known in
linear algebra. Here we use this tool to show that the determinant of a member of the subfamily of matrices with entries in $\mathcal{H}$ is a product of points of this triangle. The geometry of the triangle helps us see these properties very clearly. We now define the family $T(n,k)$ of
skew-triangular matrices $\mathcal{H}$. If $k=2,3,\ldots,n+1$, then
 \begin{equation} T(n,k)=[a_{ij}]_{1\le i,j\le n} \quad \text{ where } \quad a_{ij}=\begin{cases}
      F_{i}F_{n-j+1}, & \text{ if } k\le i+j\le n+1;\\
     0, & \text{ otherwise} .
   \end{cases}
\end{equation}

From linear algebra we know that if $\lambda_1, \dots,\lambda_n$ are the eigenvalues of a matrix $A$, then the determinant of $A$ is given by,
$\det(A)= \prod_{i=1}^{n}\lambda_i$ (see \cite{Zhang}).  It is easy to verify that for a fixed $n$ and
$i\ne j \in \{2,\dots, n+1\}$, the matrices $T(n,i)$ and $T(n,j)$ do not necessarily have the same eigenvalues
(in most of the cases the eigenvalues are not integers). Proposition \ref{proposition:determinant} shows that if $n$ is fixed,
the product of the eigenvalues of $T(n,k)$ is equal to the product of eigenvalues of $T(n,n+1)$ for $1<k\le n$.
This is the product of points located in the ``median" of  $\mathcal{H}$.

\begin{proposition}\label{proposition:determinant}
If $2\le k \le n+1$, then for every $n\ge 2$ this holds
\[\det (T(n,k))=\begin{cases}
\quad \prod_{i=1}^{n}F_{i}^{2},& \text{ if } n\equiv 0 \text{ or }1 \bmod 4;\\
-\prod_{i=1}^{n}F_{i}^{2},& \text{ if } n\equiv 2 \text{ or } 3\bmod 4.
\end{cases}
\]
  \end{proposition}

\begin{proof} This is straightforward using cofactors or using Leibnitz's formula of the sum over all permutations of the numbers $1,2,\ldots,n$.
\end{proof}

Note that $\det (T(n,k))=\det(A)$ where $A$ is the $n\times n$ antidiagonal matrix defined in Section \ref{antidiagonal} on page \pageref{antidiagonal}.
This result can be extended to matrices which have entries $a_{ij}=F_{i}F_{n-j+1}$ such that $n-k\le i+j\le n+1$  and
$k$ is either $1$ or $2$ or any positive integer less than $n-2$. For example, if $n=4$  and $\det(A)$ is denoted by $|A|$ then it holds that,

\begin{equation}\label{antitriangular:mat}
\begin{vmatrix}
                    F_{1}F_{4} &F_{1}F_{3} &F_{1}F_{2}& F_{1}^2\\
  F_{2}F_{4} &F_{2}F_{3} &F_{2}^2 & \\
F_{3}F_{4}&F_{3}^2& &\\
F_{4}^2& &\bigzero& \\
                \end{vmatrix}=
\begin{vmatrix}
                    0 & F_{1}F_{3}&F_{1}F_{2}& F_{1}^2\\
   F_{2}F_{4}&F_{2}F_{3} &F_{2}^2 & \\
F_{3}F_{4}&F_{3}^2& &\\
F_{4}^2&&\bigzero &\\
                \end{vmatrix}=\cdots=
                \begin{vmatrix}
                    0 &0 &0& F_{1}^2\\
   0&0 &F_{2}^2 & \\
0&F_{3}^2&  & \\
F_{4}^2&& \bigzero&\\
                \end{vmatrix}=\prod_{i=1}^{4}F_{i}^{2}.
\end{equation}

\section{Acknowledgement}  	
All the authors of this paper were partially supported by The Citadel Foundation. The authors would like to thank
Dr. Mei Chen for the numerous insightful discussions they had with her on various topics from linear algebra,
these discussions helped improve the paper.

\bigskip
\hrule
\bigskip
	
\noindent 2010 {\it Mathematics Subject Classification}:
Primary 11B39, 15A03, 15A18; Secondary 11B83.
	
\noindent \emph{Keywords: }
Fibonacci number, Hosoya triangle,  matrix, eigenvalue, and eigenvector.
	
\end{document}